\documentclass{amsart}
\usepackage[T1]{fontenc} 

\usepackage{amssymb}
\usepackage{amsthm}
\usepackage{amsmath}
\usepackage{xcolor}
\usepackage[matrix, arrow]{xy}
\xyoption{arrow}
\theoremstyle{plain}\newtheorem{Theorem}{Theorem}
\theoremstyle{plain}
\theoremstyle{plain}
\theoremstyle{plain}
\theoremstyle{plain}
\theoremstyle{definition}
\theoremstyle{definition}
\theoremstyle{definition}
\theoremstyle{definition}
\theoremstyle{definition}
\theoremstyle{definition}
\theoremstyle{definition}

\def\Br{\mathrm{Br}}

\makeindex
\title{A block theoretic proof of Thompson's $A\times B$-Lemma} 
\author{Radha Kessar and Markus Linckelmann} 
\date{\today}

\begin{document}

\begin{abstract}
We show that Thompson's $A\times B$-Lemma can be obtained as a consequence
of  the Brauer pair version of Brauer's Third Main Theorem.
\end{abstract}

\maketitle

Let $k$ be an algebraically closed field of prime characteristic $p$.
Brauer's Third Main Theorem \cite[Theorem 3]{Brauer64}, rephrased using Brauer pairs 
(cf. \cite[Theorem 3.13]{AlBr} or 
\cite[Theorem 6.3.14]{LiBookII}), states that if $b$ is the principal block idempotent of
a finite group algebra $kG$, then $\Br_Q(b)$ is the principal block idempotent of 
$kC_G(Q)$,  for any $p$-subgroup $Q$ of $G$. Here $\Br_Q : (kG)^Q\to kC_G(Q)$
denotes the Brauer homomorphism (cf. \cite[\S 1.2]{Broue79} or \cite[Theorem 5.4.1]{LiBookI}).
In particular, if $kG$ has a unique block, 
then $kC_G(Q)$ has a unique block.  We use this to give a proof of the following result.

\begin{Theorem}[{Thompson's $A \times B$-Lemma; cf. \cite[Chapter 5, Theorem 3.4]{Gor}}]
Let $A \times B$ be a subgroup of the automorphism group of a finite $p$-group $P$,
with $A$ a $p'$-group and $B$ a $p$-group. If $A$ acts trivially on $C_P(B)$, then $A=1$.
\end{Theorem}

\begin{proof}
Consider the group $G=P \rtimes (A\times B)$, where the notation is as in the
statement, and suppose that $A$ acts trivially on $C_P(B)$. Note that $S=P\rtimes B$
is the unique, and hence normal, Sylow $p$-subgroup of $G$, and that $A$ can be regarded
as a $p'$-subgroup of the automorphism group of $S$. Thus $kG$ has a unique block
(cf. \cite[Corollary 6.2.9]{LiBookII}). 

By the assumptions, $A$ acts trivially on the group $Q = C_P(B) \times B$.
That is, we have $A \leq C_G(Q)$, and hence $C_G(Q) = C_S(Q) \rtimes A$.
By the above, $kC_G(Q)$ has a unique block.
We show now that $Q$ is self-centralising in $S$.
Let $x\in C_S(Q)$. Write $x = yu$ for some $y\in P$ and $u\in B$. Since 
$u\in B$, it follows that conjugation by $u$ preserves the decomposition
$Q = C_P(B) \times B$. Thus conjugation by $y$ preserves this decomposition as well.
In  particular, $y$ normalises $B$. By elementary group theory, it follows that
$y$ centralises $B$. Indeed,
if $u\in$ $B$, then $yuy^{-1}u^{-1}  \in P \cap B  =1 $.
This shows that
$C_S(Q) \leq Q \leq C_G(A)$, so $C_G(Q) = C_S(Q) \times A$. But then $kC_G(Q)$
has as many blocks as $A$ has irreducible characters, and so we conclude that $A=1$.
\end{proof}


\end{document}